\documentclass[11 pt,reqno]{amsart}

\usepackage{amsmath,amsthm,amssymb,amscd,graphicx}
\usepackage{bbm, dsfont}
\usepackage{color}
\usepackage{hyperref}
\usepackage{enumerate}
\usepackage{mathtools}
\usepackage{mathrsfs}
\usepackage{multirow}

\newtheorem{thm}{Theorem}[section]
 \newtheorem{cor}[thm]{Corollary}
  \newtheorem{assumption}[thm]{Assumption}
    
 \newtheorem{lem}[thm]{Lemma}
 
 \newtheorem{prop}[thm]{Proposition}
 
 \theoremstyle{definition}
 
 \theoremstyle{remark}
 \newtheorem{rem}[thm]{Remark}
 
 \numberwithin{equation}{section}

\def\be#1 {\begin{equation} \label{#1}}
\newcommand{\ee}{\end{equation}}
\renewcommand{\phi}{\varphi}

\def\C{\mathbb C}
\def\R{\mathbb R}
\def\T{\mathbb T}
\def\Z{\mathbb Z}
\def\HH{\mathbb H}
\def\N{\mathbb N}
\def\E{\mathcal E}

\def\eps{\epsilon}
\def\dis{\displaystyle}    

\DeclarePairedDelimiter{\ceil}{\lceil}{\rceil}
\definecolor{gr}{rgb}   {0.,   0.69,   0.23 }
\definecolor{bl}{rgb}   {0.,   0.5,   1. }
\definecolor{mg}{rgb}   {0.85,  0.,    0.85}
\definecolor{yl}{rgb}   {0.8,  0.7,   0.}
\definecolor{or}{rgb}  {0.7,0.2,0.2}

\newcommand{\wt}{\widetilde} 
\renewcommand{\L}{\mathscr{L}}

\catcode`\ = \active\def {\'e} \catcode`\ = \active\def {\'E}  \catcode`\Ë = \active\def Ë{\`A}\catcode`\ = \active\def {\`e}
\catcode`\ = \active\def {\c{c}} \catcode`\ = \active\def
{\`a} \catcode`\ = \active \def {\^e} \catcode`\ = \active
\def {\`u} \catcode`\ = \active \def {\^o} \catcode`\ =
\active \def {\^u} \catcode`\ = \active \def {\^{\i}}
\catcode`\ = \active \def {\^a} \catcode`\ = \active \def
{{\"o}}

\newcommand{\om}{  \omega   }

\newcommand{\ov}{  \overline  }

\newcommand\<{\langle}
\renewcommand\>{\rangle}

\begin{document}

\thanks{The author is supported by the grants   "BEKAM''  ANR-15-CE40-0001 and  "ISDEEC'' ANR-16-CE40-0013}

\author{ Laurent Thomann }
\address{Institut  \'Elie Cartan, Universit\'e de Lorraine, B.P. 70239,
F-54506 Vand\oe uvre-l\`es-Nancy Cedex, FR}
\email{laurent.thomann@univ-lorraine.fr}

\title{Growth of Sobolev norms for  linear Schr\"odinger operators}



\begin{abstract}
 We give an example of a  linear, time-dependent, Schr\"odinger operator with optimal growth of Sobolev norms. The construction is explicit, and relies on  a comprehensive study of the linear Lowest Landau Level equation with a time-dependent potential. 
 \end{abstract}

\maketitle

%

\section{Introduction and main result}

 The aim of this paper is to present an example of linear, time-dependent, Schr\"odinger operator which exhibits optimal polynomial growth of Sobolev norms.  Moreover, this operator takes the form $\wt{H}+\mathscr{L}(t)$, where $\wt{H}$ is an elliptic operator with compact resolvent and where the perturbation $\mathscr{L}(t)$ is a   small, time-dependent, bounded self-adjoint operator. Our construction is actually entirely explicit and it is based on the study of linear Lowest Landau Level equations (LLL) with a time-dependent potential. \medskip

In   Maspero-Robert~\cite{MasRo}, the authors study  linear Schr\"odinger operators, obtain global well-posedness results and prove very precise polynomial bounds on the possible growth of Sobolev norms under  general conditions (see Assumption~\ref{assumption1} below).  We show here that these bounds are optimal. \medskip

Our setting is the following: consider the 2-dimensional harmonic oscillator
$$
H =-(\partial^2_x+\partial^2_y)+(x^2+y^2)= -4\partial_z \partial_{\ov z}+|z|^2,
$$
where  $z=x+iy$,  $\partial_z=\frac12(\partial_x-i\partial_y)$. This operator acts on the space
$$
\wt{\mathcal{E}} =\big \{\, u(z) = e^{-\frac{|z|^2}{2}} f(z)\,,\;f \; \mbox{entire\ holomorphic}\,\big\}\cap \mathscr{S}'(\C),
$$
and if we define the  Bargmann-Fock space $\mathcal{E}$ by
$$
\mathcal{E} =\big \{\, u(z) = e^{-\frac{|z|^2}{2}} f(z)\,,\;f \; \mbox{entire\ holomorphic}\,\big\}\cap L^2(\C ),
$$
then the so-called    special Hermite functions   $(\phi_n)_{n \geq 0}$    given by
\begin{equation*} 
\varphi_n(z) = \frac{z^n}{\sqrt{\pi n!}}  e^{-\frac{|z|^2}{2}},
\end{equation*}
form    a Hilbertian basis of $\mathcal{E}$, and   are  eigenfunctions of $H$, namely 
$$H\phi_n=2(n+1)\phi_n, \quad n\geq 0.$$~

Let  $0\leq \tau <1$ and  set $\rho(\tau)=\frac1{2(1-\tau)} \in [1/2 ,\infty)$. We define the operator $\wt{H}_\tau=(H+1)^{\rho(\tau)}$, which in turn  defines the scale  of Hilbert spaces $\big(\widetilde{\HH}^{s}_\tau\big)_{s\geq 0}$ by 
\begin{equation*} 
\widetilde{\HH}_\tau^{s} = \big\{ u\in L^2(\C),\; \wt{H}_\tau^{s/2}u\in L^2(\C)\big\} \cap \E,\qquad \widetilde{\HH}_\tau^{0} =\E,
\end{equation*}
and we   denote by $L$  the Lebesgue measure on $\C$.\medskip

For $X$ a Banach space, we denote by  $\mathcal{C}_b\big(\R;X\big)$ the subspace of $\mathcal{C}\big(\R;X\big)$ composed of bounded functions:
$$\mathcal{C}_b\big(\R;X\big)= \big\{t\mapsto u(t) \in \mathcal{C}\big(\R;X\big) :\;  \sup_{t \in \R} \|u(t)\|_X <+\infty     \big\}.$$
Similarly, for all $k \in \N$ we define the space $\mathcal{C}^k_b\big(\R,X\big)$ by 
$$\mathcal{C}^k_b\big(\R;X\big)= \big\{ u \in \mathcal{C}_b\big(\R;X\big) :\;     \partial^j_t u \in \mathcal{C}_b\big(\R;X\big),\; \forall \,0 \leq j \leq k   \big\}.$$\medskip

Let $0 \leq\tau <1$ and $s\geq 0$. For a family $\big( \mathscr{L}(t)\big)_{t\in \R}$  of continuous linear mappings
 \begin{equation*}
 \mathscr{L}(t) :  \widetilde{\HH}^{s}_\tau\longrightarrow\widetilde{\HH}_\tau^{s},
  \end{equation*}
  we consider the following assumptions :
  \begin{assumption}\label{assumption1} $\big( \mathscr{L}(t)\big)_{t\in \R}$ is a family of linear operators which satisfies: 
\begin{enumerate}[$(i)$]
 \item  One has $t\longmapsto  \mathscr{L}(t) \in \mathcal{C}_b\big(\R; \mathcal{L}(\widetilde{\HH}_\tau^{s})\big)$ for all $s \geq 0$. 
  \item  For every $t\in \R$, $\mathscr{L}(t)$ is symmetric  w.r.t. the scalar product of $\widetilde{\HH}_\tau^{0}$, 
  $$\int_{\C} \ov{v} \mathscr{L}(t) u  \,dL=  \int_{\C}   u\ov{ \mathscr{L}(t) v }\,dL, \quad \forall{u,v \in \widetilde{\HH}_\tau^{0}}.$$
   \item   The  family $\big( \mathscr{L}(t)\big)_{t\in \R}$   is $\wt{H}_\tau^{\tau}$-bounded in the sense that 
 $t\longmapsto   [\mathscr{L}(t), \wt{H}_\tau] \wt{H}_\tau^{-\tau} \in \mathcal{C}_b\big(\R,\mathcal{L}(\widetilde{\HH}_\tau^{s})\big)$ for all $s \geq 0$.
  \item  For all $\ell  \in \N$, one has $t\longmapsto  \mathscr{L}(t) \in \mathcal{C}^{\ell}_b\big(\R; \mathcal{L}(\widetilde{\HH}_\tau^{s} ; \widetilde{\HH}_\tau^{s- \ell \tau})\big)$ for all $s \geq 0$. 
 \end{enumerate}
\end{assumption}

 Finally, for $s\geq 0$, we consider the initial value problem
  \begin{equation} \label{lin-abs}
\left\{
\begin{aligned}
&i\partial_t u=\big(\wt{H}_\tau+\mathscr{L}(t) \big)u , \quad   (t,z)\in \R\times \C,\\
&u(t, \cdot)_{|t=t_0}=  u_0 \in \widetilde{\HH}_\tau^{s},
\end{aligned}
\right.
\end{equation} 
and we are able to state our  main result :
 \begin{thm} \label{thm-abs}
 For any $0 \leq \tau<1$ and any $\eps>0$, there exists a family of linear operators $\big( \mathscr{L}(t)\big)_{t\in \R}$ which satisfies Assumption~\ref{assumption1}, so that for all $s\geq 0$ :
 \begin{enumerate}[$(i)$]
\item There exists $C_s>0$ such that 
$$\dis \sup_{t\in \R}\| \mathscr{L}(t)  \|_{ \mathcal{L}(\widetilde{\HH}_\tau^{s})} \leq C_s \eps.$$

\item The problem \eqref{lin-abs} is globally well-posed in $\widetilde{\HH}_\tau^{s}$ : for any $u_0\in \widetilde{\HH}_\tau^{s}$, there exists a unique solution $u(t):=\mathcal{U}(t,t_0)u_0$ such that $u\in \mathcal{C}\big(\R, \widetilde{\HH}_\tau^{s}\big)$ to \eqref{lin-abs}. Moreover, $\mathcal{U}$ has the group property
$$\mathcal{U}(t_2,t_1)\mathcal{U}(t_1,t_0)=\mathcal{U}(t_2,t_0), \quad \mathcal{U}(t,t)=I_d, \quad  \forall \, t,t_1, t_2 \in  \R,$$
and $\mathcal{U}$   is unitary in $\widetilde{\HH}_\tau^{0}$
$$ \big\| \mathcal{U}(t,t_0)u_0\big\|_{\widetilde{\HH}_\tau^{0}}=\| u_0\|_{\widetilde{\HH}_\tau^{0}},\quad  \forall \, t  \in  \R.$$
\item Any solution to \eqref{lin-abs}, with initial condition $u_0 \in \widetilde{\HH}_\tau^{s}$, satisfies for all $t\in \R$
 \begin{equation*} 
\big\| \mathcal{U}(t,t_0)u_0\big\|_{\widetilde{\HH}_\tau^{s}}\leq C \| u_0\|_{\widetilde{\HH}_\tau^{s}}\<\eps(t-t_0)\>^{\frac s{2(1-\tau)}}.
\end{equation*}
\item There exists a nontrivial initial condition $u_0 \in \cap_{k\geq 1} \widetilde{\HH}_\tau^{k} $ such that   the corresponding solution to \eqref{lin-abs} satisfies   for all $t\in \R$
 \begin{equation*} 
\big\| \mathcal{U}(t,t_0)u_0\big\|_{\widetilde{\HH}_\tau^{s}} \geq c  \| u_0\|_{\widetilde{\HH}_\tau^{s}}\<\eps(t-t_0)\>^{\frac s{2(1-\tau)}}.
\end{equation*}
 \end{enumerate}
\end{thm}

 Actually, items $(ii)$ and $(iii)$ directly follow from~\cite[Theorem 1.5]{MasRo}. The novelty in our work is item    $(iv)$ which shows that the upper  bounds obtained in  \cite[Theorem 1.5]{MasRo} are optimal without further assumptions, even for small perturbations $\L(t)$, see item $(i)$. \medskip
  
 It seems that the  example of Theorem~\ref{thm-abs} is the first one which covers all the possible values of  $0\leq \tau <1$, and it is noticeable the result of Theorem~\ref{thm-abs} is obtained for any value of $0\leq \tau<1$, by essentially the same example,  written in different scales of Hilbert spaces.    An example of such growth was given in~\cite{Delort2} in the case $\tau=0$ (see also~\cite{Maspero} for an alternative proof), and in~\cite[Appendix A]{BGMR2} in the case $\tau=1/2$, but it seems that the other cases were left open. \medskip
 
 We stress that our example is an operator which  takes the form $\wt{H}_\tau+\mathscr{L}(t)$, where $\wt{H}_\tau$ is a (time independent) elliptic operator with compact resolvent and $\mathscr{L}(t)$ a bounded self-adjoint operator. Moreover, this perturbation is small and satisfies indeed for all $\ell \geq 0$
 $$\dis \sup_{t\in \R}\| \partial^{\ell}_t\mathscr{L}(t)  \|_{ \mathcal{L}(\widetilde{\HH}_\tau^{s}; \widetilde{\HH}_\tau^{s- \ell \tau})} \leq C_{s,\ell} \eps.$$~
 
 If one allows unbounded perturbations,  it is simpler to obtain growth of Sobolev norms, as it is shown by an elementary example given in Appendix~\ref{appendixA}. In this latter context, growth of Sobolev norms can occur even with time-independent operators.\medskip

   Observe that  the $\phi_n$ are the eigenfunctions of $\wt{H}_\tau$, namely 
$$\wt{H}_\tau\phi_n=2^{\rho(\tau)} (n+1)^{\rho(\tau) }\phi_n, \quad n\geq 0.$$
Hence in our example, we see an exact correspondence between the asymptotics of the eigenvalues of~$\wt{H}_\tau$ and the rate of growth for \eqref{lin-abs}. If $\rho>1$ (which corresponds to $\tau >1/2$), the operator~$\wt{H}_\tau$ satisfies a gap condition, but in our example, $\partial^{\ell}_t \mathscr{L}(t)$ is not regular enough (see item $(iv)$ in Assumption~\ref{assumption1}) to  meet the  hypotheses of~\cite[Theorems~1.8~and~1.9]{MasRo}, in which better upper bounds are obtained.
 \medskip

Let us recall the following characterization of the Sobolev spaces $\widetilde{\HH}_\tau^{s}$. By~\cite[Lemma~C.1]{GGT}, for any   $s \geq 0$,  there exist $c,C>0$ such that for all  $u\in  \wt{\HH}_\tau^s$ 
\begin{equation*} 
c\|\<z\>^{s\rho(\tau)} u\|_{L^2(\C)} \leq  \|u\|_{\wt{\HH}_\tau^s} \leq C\|\<z\>^{s\rho(\tau)}  u\|_{L^2(\C)}, \quad \<z\>=(1+|z|^2)^{1/2}.
\end{equation*}
As a consequence, in the Bargmann-Fock space, a growth of Sobolev norm corresponds to a transfer of energy in the physical space. In our example, the growth will be induced by a traveling wave. This is in contrast to the previous known examples \cite{Delort2, BGMR2, Maspero}, where the growth was inherited by a time-periodic phenomenon.  \medskip

We end this section by reviewing  some results on the growth of  linear Schr\"odinger equations on manifolds with time-dependent potentials
\begin{equation}\label{NLS} 
i\partial_t u+\Delta u  +V(t,x)u=0.
\end{equation}
In \cite{Bou99a} Bourgain  proves a polynomial bound of the Sobolev norm for \eqref{NLS}, when  $V(t,x)$ is a bounded (real analytic) potential. Moreover, when the potential is  quasi-periodic in time  he obtains  in \cite{Bou99b} a logarithmic bound (see also \cite{Delort, Wang2008, FangZhang, GPT},  for  more results on norm inflation phenomena  in various settings). Delort \cite{Delort2} constructs an example with  polynomial growth for the harmonic oscillator perturbed by a (time-periodic) pseudo-differential operator of order zero. In~\cite{BGMR2}, the authors give the example of a time-periodic order one perturbation of the harmonic oscillator which induces polynomial growth.  In~\cite{ANS}, the authors  prove exponential growth of the energy norm for a linear (and nonlinear) harmonic oscillator perturbed by the angular momentum operator (see~\cite[Theorem~4.5]{ANS}).  We refer to~\cite{Maspero, HausMaspero}  for more examples with  growth of norms and to~\cite{BGMR} for bounds on abstract linear Schr\"odinger equations. Let us mention the article~\cite{LZZ} in which the authors obtain very precise results on the dynamics of a family of perturbations of the harmonic oscillator. Finally, in the recent paper~\cite{Faou-Rapha}, Faou and Rapha\"el  give  examples of solutions to perturbed harmonic oscillators which grow like $(\log t)^\alpha$. Interestingly, although being different to ours, one of their approach relies on the  study of the so-called continuous resonant equation (CR) which contains the dynamics of the LLL equation. \medskip


\section{The linear LLL equation with time-dependent potential}\label{Sect2}

We now present our example more in details. Let $W\in L^{\infty}(\R\times \C,\R)$ be a real-valued time-dependent potential and consider the linear equation

 \begin{equation}\label{lin-LLL} 
\left\{
\begin{aligned}
&i\partial_t u- \delta H u=\Pi\big(W(t,z)u\big), \quad   (t,z)\in \R\times \C, \quad \delta \in \R,\\
&u(t, \cdot)_{|t=t_0}=  u_0 \in \E,
\end{aligned}
\right.
\end{equation}
where $\Pi$ is the orthogonal projector on the space $\E$ (the kernel of $\Pi$ is very explicit, see \eqref{ker-pi} below). The equation~\eqref{lin-LLL} is  the linearization of the Lowest  Landau Level equation
 \begin{equation}\label{eq-disp} 
i\partial_t u- \delta Hu=\Pi(|u|^2u), 
\end{equation}
which  is used in the modeling of fast rotating Bose-Einstein condensates. See e.g. the introduction of~\cite{GGT} for physical motivation, and we refer to \cite{ABN, Nier, GGT, BiBiCrEv2, BiBiEv, Schw-Tho} for the   study of \eqref{lin-LLL}. Equation~\eqref{lin-LLL} is  a natural mathematical toy model, for which we can try to exhibit some particular dynamics.  \medskip

The dispersion parameter $\delta \in \R$ does not play a role in the dynamics of equation~\eqref{eq-disp}. Actually,~$u$ solves~\eqref{eq-disp} if and only if $v=e^{i \delta t H}u$ solves~\eqref{eq-disp} with $\delta=0$. This comes from the crucial property 
\begin{equation*}
e^{-itH}\Pi\big(e^{itH} a\, \ov{e^{itH} b}\, e^{itH} c\big)= \Pi\big( a\, \ov{b}\, c\big), \quad  \forall \,a,b,c  \in \E,
\end{equation*}
see \cite[Lemma~2.4 and Corollary~2.5]{GHT1}. However, the  transformation $v=e^{i \delta t H}u$ does not preserve the left hand side of \eqref{lin-LLL}, that is why we must keep  the parameter $\delta \in \R$  in our study (nonetheless we will see that it does not affect the dynamics of equation~\eqref{lin-LLL}, excepted in the reducibility result stated in Appendix~\ref{appendixA} where we need $\delta \neq 0$). \medskip

In the sequel, by a time translation, we  restrict to the case $t_0=0$. \medskip

In this section,  we state global well-posedness results with optimal bounds on the growth of the Sobolev norms for~\eqref{lin-LLL}. We are also able to obtain reducibility results for~\eqref{lin-LLL}, when $W$ is a small quasi-periodic potential, but these results are direct applications of~\cite{Grebert-Thomann}, thus we have postponed the statements to the Appendix~\ref{appendixA}.

\subsection{Statement of the results} Our first result concerns the global well-posedness of such an equation under general conditions on $W$. For $s\geq 0$, we denote by 
\begin{equation}\label{defl2}
L^{2,s}=\big\{u\in L^2(\C), \;\<z\>^su \in L^2(\C)\big\}, \quad \<z\>=(1+|z|^2)^{1/2}
\end{equation}
 the weighted Lebesgue space and $L^{2,s}_{\E}=L^{2,s}  \cap \E$. Then our well-posedness result reads:
 \begin{thm}\label{thm2.1}
Let $\delta \in \R$ and $W\in L^{\infty}(\R\times \C,\R)$.  For all $u_0 \in \mathcal{E}$, there exists a unique solution $u\in \mathcal{C}  (\R ,  \mathcal{E})$ to equation \eqref{lin-LLL}. Moreover, for every  $t\in \R$, 
 $$\int_{\C}|u(t,z)|^2dL(z)=\int_{\C}|u_0(z)|^2dL(z).$$
Furthermore, if for some $s>0$, $u_0 \in L^{2,s}_\mathcal{E}$, then  $u(t) \in L^{2,s}_\mathcal{E}$ for every $t\in \R$.
\end{thm}

 A natural question is the control of higher order Sobolev norms of the solution for large times and  this will be achieved, under some additional conditions on $W$. The notation $W \in \mathcal{C}^{\infty}_{b_t}\big(\R\times  \R^2, \R\big)$ means that $W$ is continuous and bounded in $t$ and smooth in the variables~$(x,y)$. We stress that derivation in the time variable is not needed. For notational convenience,  we sometimes identify $(x,y) \in \R^2$ and $z=x+iy \in \C$. In particular, for a function of the variables $(x,y)$ we use the notations 
\begin{equation}\label{not-der}
\partial_z=\frac12(\partial_x-i\partial_y), \qquad \partial_{\ov{z}}=\frac12(\partial_x+i\partial_y).
\end{equation}

\begin{thm} \label{thm2.2}
Let $\delta \in \R$ and $s\geq 0$. Assume that  $W \in \mathcal{C}^{\infty}_{b_t}\big(\R\times  \R^2, \R\big)$ is such that  
 \begin{equation}\label{condi-W}
 \sup_{0 \leq k\leq \ceil[\big]s} \|\partial^k_z W(t,\cdot)\|_{L^{\infty}(\C)}\leq C_0, \quad t\in \R,
   \end{equation}
then any solution to \eqref{lin-LLL}, with initial condition $u_0 \in L^{2,s}_\mathcal{E}$, satisfies for all $t\in \R$
 \begin{equation} \label{condi-major}
\|\<z\>^s u(t)\|_{L^{2}(\C)} \leq C\|\<z\>^s u_0\|_{L^{2}(\C)}  \<C_0 t\>^{s},   
\end{equation}
where the constant $C>0$ only depends on $s\ge 0$.
\end{thm}

Condition~\eqref{condi-W} is rather natural in the space $\E$. For instance, it is satisfied by the following class of potentials: assume that $V(t, \cdot) \in \E$, uniformly in $t\in \R$, then $W=|V|^2$ satisfies \eqref{condi-W} for all $k \geq 0$, by Lemma~\ref{lem.deri} and~\eqref{hyper}. We stress that in~\eqref{condi-W} one needs the operator $\partial_z$ (not $\partial_x$ or $\partial_y$).

\medskip

This  bound is indeed optimal as shown by the next result:

 \begin{thm} \label{thm2.3}
Let $\delta \in \R$.  For all $\eps>0$, there exists $W_\eps \in \mathscr{S}(\R \times \R^2, \R)$ such that  for all $1 \leq p \leq \infty$, $\|W_\eps(t,\cdot)\|_{L^p(\C)}\leq \eps$, and such that for all $k,j\geq 0$ and  uniformly in time 
 \begin{equation}\label{condi-der}
  \|\partial^j_{\ov z }  \partial^k_z W_\eps(t,\cdot)\|_{L^{\infty}(\C)}\leq  \eps  C_{jk}, \quad t\in \R,
 \end{equation}
and there exists a nontrivial initial condition $ u_0  \in \bigcap_{k\geq 0} L^{2,k}_\mathcal{E}$  such that the corresponding solution to \eqref{lin-LLL} satisfies for all $s \geq 0$ and $t\in \R$
 \begin{equation}\label{condi-minor}
 \|\<z\>^s u(t)\|_{L^{2}(\C)} \geq c_s  \|\<z\>^s u_0\|_{L^{2}(\C)}\<\eps t\>^{s}.
  \end{equation}
 Moreover, we have the following equivalence, when $ t\longrightarrow \pm \infty$
\begin{equation}\label{eq-l2}
\|\<z\>^s u(t)\|_{L^{2}(\C)}   \sim c^s \eps^s |t|^s  \|u_0\|_{L^{2}(\C)}.
\end{equation}
\end{thm}

The inequality \eqref{condi-minor} is stated like that in order to give a counterpart to \eqref{condi-major}. Actually, as shown in \eqref{eq-l2}, the coefficient of the leading order of the lower term depends on  $\|u_0\|_{L^{2}(\C)}$ (and not on $ \|\<z\>^s u_0\|_{L^{2}(\C)}$). The inequality  \eqref{condi-minor} holds true since $u_0  \in \bigcap_{k\geq 0} L^{2,k}_\mathcal{E}$ is here fixed, and the constant $c_s>0$ depends on $u_0$. \medskip

  Observe that using the notations \eqref{not-der}, the condition \eqref{condi-der} can be rephrased as 
  \begin{equation*} 
  \|\partial^j_{x }  \partial^k_y W_\eps(t,\cdot)\|_{L^{\infty}(\C)}\leq  \eps  C_{jk}, \quad t\in \R.
 \end{equation*}

The result of Theorem \ref{thm2.3} is a direct consequence of  \cite[Theorem~1.5 and Corollary~1.6]{Schw-Tho} and we can make the explicit choices 
$$ \dis u_0=\sqrt{\eps} \big(\frac12\phi_0+i\frac{\sqrt{3}}2\phi_1\big), \quad \alpha=\frac{\sqrt{3}}{32 \pi}\eps,$$
 and
 \begin{equation}\label{def-W}
   W_\eps(t,z)=\frac{\eps}{4\pi}\big|  1-i\sqrt{3}(e^{-2i \delta t}z+\alpha t) \big|^2e^{-|e^{-2i \delta t}z+\alpha t|^2}.
   \end{equation}
 Actually, in   \cite[Theorem~1.5 and Corollary~1.6]{Schw-Tho} (see also \cite[equation (2.4)]{Schw-Tho}), unbounded trajectories where constructed for the system   
 \begin{equation}\label{sys-conj}
\left\{
\begin{aligned}
&i\partial_{t}{u}-\delta H{u} =    \Pi (|{v}|^2 {u}), \quad   (t,z)\in \R\times \C,\\
&i\partial_{t}{v}-\delta H{v}=  - \Pi (|{u}|^2 {v}),\\
&{u}(0,\cdot)=  u_0,\; {v}(0,\cdot)=  v_0,
\end{aligned}
\right.
\end{equation}
and the idea is here to consider the second equation in \eqref{sys-conj} as given, and to interpret the term~$|v|^2$ in the first line as a given time-dependent potential. \medskip

  Notice that the growth of Sobolev norms is not obtained by a periodic potential as in \cite{Delort,BGMR2}. Here, as it is shown in \eqref{def-W}, the growth is exhibited by a time translation (more precisely, by a magnetic translation in the Bargmann-Fock space). \medskip
  
  In general, growth of Sobolev norms is a phenomenon which happens due to resonances of the equation. Recall that  the dynamics of the cubic LLL equation $i \partial_t u = \Pi(|u|^2u)$ is included in the so-called cubic resonant (CR) equation, which was derived in~\cite{FGH} as a resonant approximation of NLS (we also refer to~\cite{GHT1} for a comprehensive study of the (CR) equation). \medskip

 In the last result of this section we show that if $W$ has additional spacial decay, then the possible growth of the solution of \eqref{lin-LLL} enjoys better controls  :

\begin{thm} \label{thm2.4}
Let $\delta \in \R$ and $s\geq 0$ and let $W \in \mathcal{C}^{\infty}_{b_t}\big(\R\times  \R^2, \R\big)$.
\begin{enumerate}[$(i)$]
\item  Assume that    
 \begin{equation*} 
  \sup_{1 \leq j \leq  k\leq \ceil[\big]s}     \|z^{2k-j}\partial^j_z W(t,\cdot)\|_{L^{\infty}(\C)}\leq C_{1}, \quad t\in \R,
   \end{equation*}
then any solution to \eqref{lin-LLL}, with initial condition $u_0 \in L^{2,s}_\mathcal{E}$, satisfies for all $t\in \R$
 \begin{equation*} 
\|\<z\>^s u(t)\|_{L^{2}(\C)} \leq C\|\<z\>^s u_0\|_{L^{2}(\C)}  \< C_1 t\>^{1/2},   
\end{equation*}
where the constant $C>0$ only depends on $s\ge 0$.
\item  Let $\eps>0$. Assume that    for all $k \geq 1 $ and  uniformly in time   
 \begin{equation}\label{condi-W2}
     \|\<z\>^{2k}\partial^k_z W(t,\cdot)\|_{L^{\infty}(\C)}\leq C_{k}, \quad t\in \R,
   \end{equation}
then any solution to \eqref{lin-LLL}, with initial condition $u_0 \in L^{2,s}_\mathcal{E}$, satisfies for all $t\in \R$
 \begin{equation*} 
\|\<z\>^s u(t)\|_{L^{2}(\C)} \leq C\|\<z\>^s u_0\|_{L^{2}(\C)}  \<  t\>^{\eps},
\end{equation*}
where  the constant $C>0$ depends on $W$, $s\geq 0$ and $\eps>0$.
\end{enumerate}
\end{thm}

 This result is the analogous to \cite{Bou99a, Delort} in which similar bounds are obtained for the linear Schr\"odinger equation with time-dependent potential, but in our case the proof is much simpler. \medskip
 
 The result of Theorem~\ref{thm2.4} shows that growth of Sobolev norms can occur only if $W$ is concentrated in the region $|z| \gg 1$ when $t  \longrightarrow \pm \infty$. This is typically the case with the example of the traveling wave exhibited in Theorem~\ref{thm2.3}  (see~\eqref{def-W}). \medskip

Under additional conditions on $W$ (analyticity in time and quasi-periodicity) one can  show the solutions are indeed bounded, see  Theorem~\ref{theo-lin-LLL}.

\subsection{Plan of the paper}  The rest of the paper is organized as follows. We end this section by giving some notations. In Section~\ref{Sect3} we study the linear LLL equation \eqref{lin-LLL}. We are then able to apply these results to prove Theorem~\ref{thm-abs} in Section~\ref{Sect4}. In Appendix~\ref{appendixA} we give an another example of Schr\"odinger operator with unbounded orbits and  in Appendix~\ref{appendixB} we state a reducibility result for~\eqref{lin-LLL}.

\subsection{Some recalls and notations}

 The harmonic oscillator $H$ is defined by
$$
H = -4\partial_z \partial_{\ov z}+|z|^2,
$$
with the classical notations $z=x+iy$ and 
$$\partial_z=\frac12(\partial_x-i\partial_y), \qquad \partial_{\ov{z}}=\frac12(\partial_x+i\partial_y).$$ 
Recall that the family of the special Hermite functions $(\phi_n)_{n \geq 0}$ is given by 
$$
\varphi_n(z) = \frac{z^n}{\sqrt{\pi n!}}  e^{-\frac{|z|^2}{2}}.
$$
The family $(\phi_n)_{n \geq 0}$ forms  a Hilbertian basis of $\mathcal{E}$ (see \cite[Proposition 2.1]{Zhu}),  and the $\phi_n$ are the eigenfunctions of $H$, 
$$H\phi_n=2(n+1)\phi_n, \quad n\geq 0.$$
 We can show (see \cite{GGT}) that $\Pi$, the orthogonal projection on $\E$, is given by the formula
\begin{equation}\label{ker-pi}
(\Pi u)(z) = \frac{1}{\pi} e^{-\frac{|z|^2}{2}} \int_\mathbb{C} e^{\ov  w z - \frac{|w|^2}{2}} u(w) \,dL(w),
\end{equation}
where $L$ stands for Lebesgue measure on $\C$. 

 Recall (see \eqref{defl2}) that for $s\geq 0$, the weighted Lebesgue    space $L^{2,s}$ is defined by
$$L^{2,s}=\big\{u\in L^2(\C), \;\<z\>^su \in L^2(\C)\big\}, \quad \<z\>=(1+|z|^2)^{1/2}$$
   and $L^{2,s}_{\E}=L^{2,s}  \cap \E$.  For $s \geq 0$, we define the harmonic Sobolev spaces     by 
\begin{equation*} 
\HH^{s} = \big\{ u\in L^2(\C),\; {H}^{s/2}u\in L^2(\C)\big\}\cap \E,
\end{equation*}
equipped with the natural norm $ \|u\|_{\HH^s} =\|H^{s/2} u\|_{L^2(\C)}$. Then by~\cite[Lemma~C.1]{GGT}, we have $\HH^{s}  = L^{2,s}_{\E}$  with the equivalence of norms
\begin{equation} \label{eqiv-0}
c\|\<z\>^s u\|_{L^2(\C)} \leq  \|u\|_{\HH^s} \leq C\|\<z\>^s u\|_{L^2(\C)}, \quad \forall\,u \in L^{2,s}_{\E}.
\end{equation}
  \medskip
 
Recall the hypercontractivity estimates (see~\cite{Carlen} or \cite[Lemma~A.2]{Schw-Tho} for the bounds without the optimal constants which will be enough for our purpose) : for all $1 \leq p \leq q \leq +\infty$ and $u \in \wt{\E}$
\begin{equation}  \label{hyper}
\left( \frac{q}{2\pi} \right)^{1/q} \| u \|_{L^q(\C)} \leq \left( \frac{p}{2\pi} \right)^{1/p} \| u \|_{L^p(\C)}.
\end{equation} ~

 In this paper $c,C>0$ denote universal constants the value of which may change from line to line.


 \section{Study of the linear LLL equation}\label{Sect3}

    \subsection{Global existence}
      To solve  equation \eqref{lin-LLL} we find a fixed point in a ball of $\E$ to 
$$F: u \longmapsto e^{-i \delta t H}u_0-i \int_0^te^{-i\delta (t-s) H} \big(\Pi(Wu)(s)\big)ds.$$
Let us sketch the proof: since $e^{i\tau  H}$ is unitary in $L^2$, we have
\begin{eqnarray*}
\|F(u)(t)\|_{L^2} &\leq& \|u_0\|_{L^2}+ \int_{0}^t \|\Pi(Wu)(s)\|_{L^2} ds \\
&\leq& \|u_0\|_{L^2}+ Ct  \sup_{s \in [0,t]} \|u(s)\|_{L^2}      \|W\|_{L^\infty},
\end{eqnarray*}
 where we used the continuity of $\Pi$ in $L^2$ in the last line (for continuity results for $\Pi$ we refer to~\cite[Proposition 3.1]{GGT}). Contraction estimates are obtained similarly, and this gives a local in time solution. Globalization can be obtained by the Gr\"onwall inequality since the equation is linear.   The $L^2$ norm of a solution is a conserved quantity, since the potential $W$ is real valued.\medskip

If moreover   $u_0 \in L^{2,s}_\mathcal{E}$, we can  prove the    wellposedness in $L^{2,s}_\mathcal{E}$, thanks to the following lemma, which we quote for future reference:

\begin{lem}
Let $W \in L^{\infty}(\C)$ and $v \in  L^{2,s}_\mathcal{E}$, then 
\begin{equation}\label{conti-zs}
\| \langle z \rangle^s \Pi \big( Wv\big) \|_{L^2(\C)}  \leq C \| W\|_{L^\infty(\C)} \|\langle z \rangle^s v\|_{L^2(\C)},
\end{equation}
and  
\begin{equation}\label{conti-flot}
\| \langle z \rangle^s e^{i\tau  H} v  \|_{L^2}  = \| \langle z \rangle^s  v  \|_{L^2}.
\end{equation}
\end{lem}

\begin{proof}
The bound \eqref{conti-zs} is a consequence of~\cite[Proposition 3.1]{GGT}. For~\eqref{conti-flot}, we first observe that for all $u \in \E=L^{2,0}_\mathcal{E}$, we have $e^{i \tau H}u(z)=e^{2i\tau }u(e^{2i\tau}z)$, as can be seen  by testing on the complete  family $(\phi_n)_{n\geq 0}$. Then~\eqref{conti-flot} follows from the change of variables $z \mapsto e^{-2i\tau }z$. 
\end{proof}

\subsection{Bounds on Sobolev norms: proof of Theorem~\ref{thm2.2}}  
  
Now that equation \eqref{lin-LLL} is well-posed, let us inspect the behaviour of the norms of the solutions. For this we need a result, which is an consequence of~\cite[Lemma 2.1]{Schw-Tho} :

\begin{lem} 
Let $k \in \N$ and let $W \in \mathcal{C}^k(\R \times \R^2, \R)$ be a real valued function. Assume that $u\in L^{2,k}_\mathcal{E}$ satisfies
$$i \partial_t u- \delta Hu=\Pi \big(W u\big).$$
Then 
\begin{equation}\label{form-deri0}
\frac{d}{dt}\int_{\C}|z|^{2k}|u(t,z)|^2dL(z)= -2\sum_{j=1}^k(-1)^j{{k}\choose{j}}\mathfrak{Im} \int_{\C}    z^k    \ov{z}^{k-j}  |u(t,z)|^2 \big( \partial_z^j     W(t,z)\big)dL(z).
\end{equation}
\end{lem}

\begin{proof}
We compute
\begin{eqnarray*}
\frac{d}{dt}\int_{\C}|z|^{2k}|u|^2dL&=& 2\mathfrak{Re} \int_{\C} |z|^{2k} \ov{u}\partial_t u dL\\
&=& 2\mathfrak{Im} \int_{\C} |z|^{2k} \ov{u}\Pi(Wu) dL+2\delta \mathfrak{Im} \int_{\C} |z|^{2k} \ov{u}Hu dL.
\end{eqnarray*}
Let us first show that  $\dis \mathfrak{Im} \int_{\C} |z|^{2k} \ov{u}Hu dL=0$. Since $H = -4\partial_z \partial_{\ov z}+|z|^2$, it remains to show that $\dis \mathfrak{Im} \int_{\C} |z|^{2k} \ov{u} \partial_z \partial_{\ov z}u dL=0$. Write $u(z)=f(z)e^{-\frac12|z|^2}$, then 
\begin{eqnarray*}
 \mathfrak{Im} \int_{\C} |z|^{2k} \ov{u} \partial_z \partial_{\ov z}u dL&=&  \mathfrak{Im} \int_{\C} |z|^{2k} \ov{f}e^{-\frac12|z|^2} \partial_z \partial_{\ov z}\big(fe^{-\frac12|z|^2} )dL\\
&=& -\frac12 \mathfrak{Im} \int_{\C} |z|^{2k} \ov{f} \big(f+z\partial_zf-\frac12 |z|^2f )e^{-|z|^2}dL\\
&=& -\frac12 \mathfrak{Im} \int_{\C} z^{k+1}\ov{z}^k \ov{f} (\partial_zf) e^{-|z|^2}dL\\
&=&0,
\end{eqnarray*}
by integrating by parts, hence the result. To complete the proof, we apply \cite[Lemma 2.1]{Schw-Tho}.
\end{proof}

We are now able to prove Theorem~\ref{thm2.2}. By linearity, it is enough to consider the case ${\|\<z\>^ku_0\|_{L^2(\C)}=1}$. We use the identity \eqref{form-deri0}. Then, since  $ \|\partial^j_z W\|_{L^{\infty}(\C)}\leq C_0$ for all $1\leq j \leq k$, we deduce by H\"older 
\begin{eqnarray*}
\frac{d}{dt}\int_{\C}|z|^{2k}|u|^2dL&\leq & C C_0\int_{\C}  \<z\>^{2k-1} |u|^2dL \\
&\leq & CC_0\ \Big(\int_{\C}  \<z\>^{2k} |u|^2dL\Big)^{1-\frac{1}{2k}} \Big(\int_{\C}   |u|^2dL\Big)^{\frac{1}{2k}},
\end{eqnarray*}
therefore, using the conservation of the mass,  
\begin{equation*} 
\frac{d}{dt} \big\| \<z\>^ku\big\|^2_{L^2(\C)} \leq CC_0 \big\| \<z\>^ku\big\|^{2-\frac1{k}}_{L^2(\C)},
\end{equation*}
which in turn implies, by time integration,  
\begin{equation*}
\|\<z\>^ku(t)\|_{L^2(\C)}\leq \big(  \|\<z\>^ku_0\|^{1/k}_{L^2(\C)}+ CC_0 |t|\big)^k \leq C( 1  + C_0 |t|)^k,
\end{equation*}
 hence the result when~$k$ is an integer. The general case follows by interpolation.

\subsection{Bounds on Sobolev norms: proof of Theorem~\ref{thm2.4}}  The proof is similar, excepted that now we  have better controls on $W$. 

$(i)$ Again, by interpolation, it is enough to consider the case $s=k$ is an integer. By  \eqref{form-deri0} we have
\begin{equation*}
\frac{d}{dt}\int_{\C}|z|^{2k}
|u|^2dL\leq  C C_1\int_{\C}    |u|^2dL ,
\end{equation*}
which implies the result by time  integration :  
\begin{equation*} 
 \big\| \<z\>^ku\big\|^2_{L^2(\C)} \leq  \big\| \<z\>^ku_0\big\|^2_{L^2(\C)} + CC_1|t| \|  u_0\|^2_{L^2(\C)}.
\end{equation*}

$(ii)$ We assume that the stronger condition \eqref{condi-W2} holds. We use here interpolation theory for linear operators. Fix $s>0$, $\eps >0$ and set $k \in \N$ such that $s/k<\eps$. Then from $(i)$ we have 
$$\|\<z\>^ku(t)\|_{L^2(\C)} \leq    C  \<t\>^{1/2} \|\<z\>^ku_0\|_{L^2(\C)}. $$
Next, the $L^2-$conservation yields $\|u(t)\|_{L^2(\C)} \leq       \|u_0\|_{L^2(\C)} $. Then by interpolation, we get that for all $0 \leq \theta \leq 1$
$$\|\<z\>^{\theta k}u(t)\|_{L^2(\C)} \leq    C^{\theta}  \<t\>^{\theta/2} \|\<z\>^{\theta k}u_0\|_{L^2(\C)}. $$
The result then follows by taking $\theta=s/k$.

 \subsection{Growth of Sobolev norms: proof of Theorem~\ref{thm2.3}}\label{para3.3}
 Let us define the   magnetic translations by the formula
\begin{equation*} 
R_{\alpha} :(u,v)(z)  \mapsto \big(u(z+\alpha) e^{\frac{1}{2}(\overline z \alpha - z \overline{\alpha})}, v(z+\alpha) e^{\frac{1}{2}(\overline z \alpha - z \overline{\alpha})} \big), \qquad  \alpha \in \mathbb{C},
\end{equation*}
as well as  the space rotations 
\begin{equation*} 
L_{\theta} :(u,v)(z)  \mapsto \big(u(e^{i\theta}z), v(e^{i\theta}z) \big), \qquad \theta \in \mathbb{T}.
\end{equation*}
As it can be checked on the $(\phi_n)_{n\geq 0}$, we have $e^{itH}= e^{2it}L_{2t}$ for all $t\in \R$. Now  we refer to   \cite[Section~1.7.2]{Schw-Tho}. The system
 \begin{equation*} 
\left\{
\begin{aligned}
&i\partial_{t}{u}-\delta H{u} =    \Pi (|{v}|^2 {u}), \quad   (t,z)\in \R\times \C,\\
&i\partial_{t}{v}-\delta H{v}=  - \Pi (|{u}|^2 {v}),\\
&{u}(0,z)=  u_0(z),\; {v}(0,z)=  v_0(z),
\end{aligned}
\right.
\end{equation*}
admits the following explicit solutions:
 \begin{eqnarray*} 
  (u,v)&=&\big(e^{-i\lambda t}      e^{-i \delta t H} R_{\alpha t} U , e^{-i\mu t} e^{-i\delta t H} R_{\alpha t} V  \big)\\
  &=&\big(e^{-i(\lambda +2\delta)t}      L_{-2\delta t}R_{\alpha t} U , e^{-i(\mu +2\delta)t}      L_{-2\delta t} R_{\alpha t} V  \big),
 \end{eqnarray*}
 with 
  \begin{equation*} 
  \dis U=\sqrt{\eps}\big(\frac12\phi_0+i\frac{\sqrt{3}}2\phi_1\big), \quad V=\sqrt{\eps}\big(\frac12\phi_0-i\frac{\sqrt{3}}2\phi_1\big)   ,
 \end{equation*}
and 
$$ \lambda = \frac{7\eps}{32\pi}, \quad \mu = -\frac{7\eps}{32\pi}, \quad   \alpha=\frac{\sqrt{3}}{32 \pi}\eps.$$
It remains to check that $W:=|v|^2$ and $u$ satisfy the assumptions and the conclusions of Theorem~\ref{thm2.3}.

 On the one hand, for all $t\in \R$, $\|u(t)\|_{L^2}=\sqrt{\eps}$, and for $s>0$, 
  \begin{equation}\label{nor-0}
 \|\<z\>^s u_0\|_{L^2}=   \|\<z\>^s U\|_{L^2} \leq c_s \sqrt{\eps}. 
     \end{equation}
Let us prove that there exists $c_s>0$ such that for all $t \in \R$
 \begin{equation}\label{lls}
 \|\<z\>^s u(t)\|_{L^2}  \geq c_s \sqrt{\eps} \<\eps t\>^s.
    \end{equation}
We have   
 $$\|\<z\>^s u(t)\|^2_{L^2}= \|\<z\>^s  R_{\alpha t}U\|^2_{L^2} = \|\<z-\alpha t\>^s U\|^2_{L^2}= \int_{\C} \big(1+|z-\alpha t|^2\big)^s |U(z)|^2dL(z)    .$$
Therefore,  
  \begin{equation}\label{m-m}
\|\<z\>^s u(t)\|^2_{L^2} \geq \frac12  \int_{\C} \big(1+|z-\alpha t|^{2s}\big) |U(z)|^2dL(z).
   \end{equation} 
 $\bullet$ By the triangle inequality we have
   \begin{equation*}
|\alpha t|^{2s} \leq  (|z- \alpha t|+|z|)^{2s} \leq 4^s(|z- \alpha t|^{2s}+ |z|^{2s}),
   \end{equation*} 
 which in turn implies
    \begin{equation*}
|z- \alpha t|^{2s} \geq 4^{-s}|\alpha t|^{2s} -|z|^{2s}.
   \end{equation*} 
 As a consequence, by \eqref{nor-0} and \eqref{m-m}
   \begin{eqnarray}\label{bornemin}
\|\<z\>^s u(t)\|^2_{L^2} &\geq& c  |\alpha t|^{2s} \int_{\C}   |U(z)|^2dL(z)- C  \int_{\C} |z|^{2s}  |U(z)|^2dL(z) \nonumber \\
&\geq& (c  |\alpha t|^{2s} - C ) \eps \nonumber \\
&\geq& (c  |\eps t|^{2s} - C ) \eps,
   \end{eqnarray} 
 where the constants $c,C>0$ have varied from line to line. From \eqref{bornemin} we deduce that there exists $C_0>0$ such that if $|\eps t|>C_0$ we have 
     \begin{equation*}
\|\<z\>^s u(t)\|_{L^2} \geq c\<\eps t\>^s \sqrt{\eps} .
   \end{equation*} 
 $\bullet$   In the regime $|\eps t|\leq C_0$, we use the inequality \eqref{m-m} to write
        \begin{equation*}
\|\<z\>^s u(t)\|_{L^2} \geq c   \| U\|_{L^2}=c \sqrt{\eps}    \geq c\<\eps t\>^s \sqrt{\eps}   .
   \end{equation*} 
 As a consequence, we have proven \eqref{lls}. \medskip
   
 On the other hand, we have the explicit expression  
 \begin{equation*}
   W(t,z)=\frac{\eps}{4\pi}\big|  1-i\sqrt{3}(e^{-2i\delta t}z+\alpha t) \big|^2e^{-|e^{-2i\delta t}z+\alpha t|^2}= \big|L_{-2\delta t}R_{\alpha t}V(z)\big|^2.
   \end{equation*}
Therefore we have 
$$\|W\|_{L^1}= \| L_{-2\delta t}R_{\alpha t}V\|^2_{L^2}= \|  V\|^2_{L^2} =\eps$$
 and from~\eqref{hyper} we have
 $$\|W\|_{L^\infty}= \| L_{-2\delta t}R_{\alpha t}V\|^2_{L^\infty}= \|  V\|^2_{L^\infty} \leq \eps.$$
  Moreover, from Lemma~\ref{lem.deri}, we deduce
  \begin{equation*} 
  \|\partial^j_{\ov z }  \partial^k_z W \|_{L^{\infty}(\C)}=     \|\partial^j_{\ov z }  \partial^k_z \big(|L_{-2\delta t}R_{\alpha t}V|^2\big) \|_{L^{\infty}(\C)} \leq   C_{jk}   \| L_{-2\delta t}R_{\alpha t}V\|^2_{L^\infty}  \leq \eps   C_{jk}  , 
 \end{equation*}
 which was the claim.
 \section{Proof of Theorem~\ref{thm-abs}} \label{Sect4}

 \subsection{Some notations}
 
  For $0\leq \tau <1$ we set $\rho(\tau)=\frac1{2(1-\tau)}>0$ and we define the operator $\wt{H}_\tau=(H+1)^{\rho(\tau)}$, where    $H$ is the harmonic oscillator  defined by
$$
H = -4\partial_z \partial_{\ov z}+|z|^2.
$$
We then  define the family of Hilbert spaces $\big(\widetilde{\HH}_\tau^{s}\big)_{s\geq 0}$ by 
\begin{equation*}  
 \widetilde{\HH}_\tau^{s} = \big\{ u\in L^2(\C),\; \wt{H}^{s/2}u\in L^2(\C)\big\} \cap \E,\quad \widetilde{\HH}_\tau^{0} =\E.
\end{equation*}
Recall that, by \cite[Lemma~C.1]{GGT}, we have
\begin{equation*} 
c\|\<z\>^{s\rho(\tau)} u\|_{L^2(\C)} \leq  \|u\|_{\wt{\HH}_\tau^s} \leq C\|\<z\>^{s\rho(\tau)}  u\|_{L^2(\C)}, \quad \<z\>=(1+|z|^2)^{1/2}.
\end{equation*}
Observe also that 
\begin{equation}  \label{eq-hs}
\widetilde{\HH}_\tau^{s} = {\HH}^{s \rho(\tau)}
\end{equation}
where ${\HH}^{\sigma}$ stands for the  harmonic Sobolev space based on the  harmonic oscillator $H$, and we have 
\begin{equation} \label{eq-hs1}
c\|\<z\>^{s} u\|_{L^2(\C)} \leq  \|u\|_{{\HH}^s} \leq C\|\<z\>^{s}  u\|_{L^2(\C)}.
\end{equation}~

In the sequel, in order to alleviate notations, we simply write  $\wt{H}= \wt{H}_\tau$  and  $\widetilde{\HH}^{s} = \widetilde{\HH}_\tau^{s}$.
 \subsection{Definition of the operator $ \mathscr{L}(t)$}
 
 Define the potential $W_0(t,z)$ as follows:  
$$  V=\sqrt{\eps}\big(\frac12\phi_0-i\frac{\sqrt{3}}2\phi_1\big), \quad W_0(t,z)=|R_{\alpha t}V(z)|^2=\frac{\eps}{4\pi}\big|  1-i\sqrt{3}(z+\alpha t) \big|^2e^{-|z+\alpha t|^2}, \quad \alpha=\frac{\sqrt{3}}{32 \pi} \eps,$$
and with Lemma~\ref{lem.deri}, we  show that all the derivatives of $W_0$ are bounded uniformly in $t\in \R$:  
  \begin{equation} \label{bk2}
  \|\partial^j_{\ov z }  \partial^k_z W_0(t) \|_{L^{\infty}(\C)}=     \|\partial^j_{\ov z }  \partial^k_z \big(|R_{\alpha t}V|^2\big) \|_{L^{\infty}(\C)} \leq   C_{jk}   \| R_{\alpha t}V\|^2_{L^\infty} = C_{jk}   \| V\|^2_{L^\infty}  \leq \eps   C_{jk}. 
 \end{equation}
Now we define the mapping
 \begin{equation}\label{def-L}
 \begin{array}{rccl}
 \mathscr{L}(t) :  &\widetilde{\HH}^{s}&\longrightarrow&\widetilde{\HH}^{s}\\[3pt]
\dis & u &\longmapsto &  e^{-it \wt{H}}    \Pi\big( W_0(t)e^{it \wt{H}}  u\big)= e^{-it (H+1)^{\rho}}    \Pi\big( W_0(t)e^{it (H+1)^{\rho}}  u\big)  ,
 \end{array}
 \end{equation}
and we consider the initial value problem
  \begin{equation}\label{le-pb}
\left\{
\begin{aligned}
&i\partial_t u=\big(\wt{H}+\mathscr{L}(t) \big)u , \quad   (t,z)\in \R\times \C,\\
&u(t)_{|t=t_0}=  u_0 \in \widetilde{\HH}^{s}.
\end{aligned}
\right.
\end{equation}

 \subsection{Verification of Assumption~\ref{assumption1}}\label{para4.3} We now prove that $ \mathscr{L}(t)$ satisfies the required properties. ~

$(i)$ Let us check that  $\mathscr{L} \in \mathcal{C}_b\big(\R, \mathcal{L}(\widetilde{\HH}^{s})\big)$, with  norm $\| \mathscr{L}(t)  \|_{ \mathcal{L}(\widetilde{\HH}^{s})} \leq C_s \eps$.   First, by \eqref{eq-hs} it is equivalent to show that $\mathscr{L} \in \mathcal{C}_b\big(\R, \mathcal{L}({\HH}^{s})\big)$. Then, since $e^{it {(H+1)^{\rho}}} $ is unitary in ${\HH}^{s}$, and by \eqref{eq-hs1}
  \begin{eqnarray*}
\big\| e^{-it (H+1)^{\rho}}    \Pi\big( W_0(t)e^{it (H+1)^{\rho}}  u\big) \big\|_{{\HH}^{s}}&=&\big\|     \Pi\big( W_0(t)e^{it (H+1)^{\rho}}  u\big) \big\|_{{\HH}^{s}}\\
&\leq &C \| \langle z \rangle^{ s} \Pi \big( W_0(t)e^{it (H+1)^{\rho}}  u\big) \|_{L^2(\C)}.
\end{eqnarray*}  
Next, by \eqref{conti-zs} and \eqref{conti-flot}
  \begin{eqnarray*}
 \| \langle z \rangle^{ s} \Pi \big( W_0(t)e^{it (H+1)^{\rho}}  u\big) \|_{L^2(\C)} &\leq & C \| W_0(t)  \|_{L^\infty(\C)}      \| \langle z \rangle^{ s}  e^{it (H+1)^{\rho}}  u \|_{L^2(\C)}\\
 &= & C \| W_0(t)  \|_{L^\infty(\C)}     \| \langle z \rangle^{ s}   u \|_{L^2(\C)}\\
 & \leq & C \| W_0(t)  \|_{L^\infty(\C)}     \|     u \|_{\HH^s}.
\end{eqnarray*}
Recall that $W_0(t)=|R_{\alpha t}V|^2$, where $V \in \E$, then $\| W_0(t)  \|_{L^\infty(\C)}  =\| V\|^2_{L^\infty} \leq C \eps $. Putting all the previous estimates toghether, we obtain
$$\big\| e^{-it (H+1)^{\rho}}    \Pi\big( W_0(t)e^{it (H+1)^{\rho}}  u\big) \big\|_{{\HH}^{s}} \leq C_s\eps \|  u\|_{{\HH}^{s}},$$
hence the announced bound. The time-continuity of $\mathscr{L}$  follows from the previous estimates together with the continuity of the translations for the Lebesgue measure and the fact that $e^{it \wt{H}} \in \mathcal{C}_b\big(\R, \mathcal{L}(\widetilde{\HH}^{s})\big)$.
\medskip

$(ii)$  The symmetry of $\mathscr{L}$, w.r.t. the scalar product of $\widetilde{\HH}^{0}=\E$,  is a consequence of the symmetry of~$\Pi$, the conjugation by the unitary operator $e^{it {(H+1)^{\rho}}} $  and the fact that $W_0$ is a real valued function.    \medskip

$(iii)$  Let us check that $ \big[\wt{H}, \L(t)\big]$ is $\wt{H}^{\tau}$-bounded.  By Lemma \ref{lem-comm}, the operators $H$ and $\Pi$ commute, thus 
   \begin{eqnarray*}\big[\wt{H}, \L(t)\big]\wt{H}^{-\tau}&=&\big[(H+1)^{\rho}, \L(t)\big](H+1)^{-\rho \tau}\\
   &=& e^{-it (H+1)^{\rho}}\Pi\big[(H+1)^{\rho} , W_0(t)\big](H+1)^{-\rho \tau}e^{it (H+1)^{\rho}}.
   \end{eqnarray*}
Recall that $\Pi$ is bounded in all the $\HH^s$ spaces, as well as the operators $e^{-it (H+1)^{\rho}}$. \medskip

$\bullet$ Case $s=0$. Let us first  prove that $\Pi\big[(H+1)^{\rho} , W_0(t)\big](H+1)^{-\rho \tau} : \E \longrightarrow \E$ is bounded, uniformly in $t\in \R$. For that,  we  use the Weyl-H\"ormander pseudo-differential calculus (we refer to~\cite{Robert, Helffer} or to~\cite[Chapter~3]{Parmeggiani} for a review of this theory). Denote by $ z=x_1+ix_2$,  $\xi=\xi_1+\xi_2$, and consider the metric 
\begin{equation}\label{metric} 
dx_1^2+dx_2^2+\frac{d\xi_1^2+d\xi_2^2}{1+|z|^2+|\xi|^2}.
\end{equation} 
The Planck function  associated to the metric \eqref{metric} is given by $h(x_1,x_2,\xi_1,\xi_2)=(1+|z|^2+|\xi|^2)^{-1/2}$ and  for $m \in \R$, the symbol class  $S^m$ is 
\begin{equation*} 
S^m =\Big\{ a \in \mathcal{C}^\infty(\R^4; \C)\,:  \; \big|\partial_{x_1}^{\alpha_1} \partial_{x_2}^{\alpha_2} \partial_{\xi_1}^{\beta_1}\partial_{\xi_2}^{\beta_2} a(x_1,x_2,\xi_1,\xi_2)\big|
\leq C_{\alpha,\beta}\langle |z|+|\xi|\rangle^{m-\beta_1-\beta_2}, \;\;  \forall \, \alpha, \beta \in \N^2\Big\}.
\end{equation*}
For $a \in S^m$, we define its Weyl-quantization by the formula
$$a^w(x,D)u(x)=\frac1{(2\pi)^2}\int_{\R^2}\int_{\R^2} e^{i(x-y)\cdot \xi}a(\frac{x+y}2,\xi ) u(y)dyd \xi, \quad u\in \mathscr{S}(\R^2).$$
First, using the functional calculus associated to the operator $H$, we obtain that   $(H+1)^\rho$ is a pseudo-differential operator with symbol in $S^{2\rho}$. By \eqref{bk2}, $W_0(t)\in S^0$ uniformly in $t\in \R$, and therefore   the commutator
$\big[(H+1)^{\rho}, W_0(t)\big]$ is a pseudo-differential operator with symbol in $S^{2\rho-1}$, which in turn implies that  $\big[(H+1)^{\rho}, W_0(t)\big](H+1)^{-\rho \tau}$  is a pseudo-differential operator with symbol in $S^{0}$ (because $2\rho-1-2\rho \tau=0$), hence it is bounded. \medskip

$\bullet$ The proof in the general case $s\geq 0$ is similar. \medskip

 $(iv)$ From \eqref{def-L}, a direct computation gives 
\begin{eqnarray}\label{zz}
\partial_t  \L(t)&=& e^{-it \wt{H}}    \Pi  (\partial_t W_0(t))e^{it \wt{H}} -i e^{-it \wt{H}}    \Pi  \big[\wt{H}, W_0(t)\big]e^{it \wt{H}}\nonumber \\
&=& e^{-it \wt{H}}    \Pi  (\partial_t W_0(t))e^{it \wt{H}} -i     \big[\wt{H}, \L(t)\big] .
  \end{eqnarray}
  From the expression of $W_0$, we deduce that $\sup_{t \in\R}\| \partial^{\ell}_t W_0(t)\|_{L^{\infty}(\C)} \leq C_{\ell}$ for all $\ell \geq 0$. In particular, the first term in the right hand side of~\eqref{zz} is bounded $\widetilde{\HH}^{s} \longrightarrow  \widetilde{\HH}^{s}$. In item $(iii)$ we have shown that  $[\wt{H}, \L(t)\big] :  \widetilde{\HH}^{s} \longrightarrow  \widetilde{\HH}^{s-\tau}$ is bounded,  uniformly in $t\in \R$. As a consequence,  for all $s\geq 0$, $\L \in \mathcal{C}^{1}_b\big(\R; \mathcal{L}(\widetilde{\HH}^{s} ; \widetilde{\HH}^{s- \tau})\big)$. The general case $\ell\geq 0$ is obtained by induction.

\subsection{Proof of  Theorem~\ref{thm-abs}} We are now ready to complete the proof of Theorem~\ref{thm-abs}. Consider the problem~\eqref{le-pb} and for convenience, assume that $t_0=0$. By~\eqref{def-L}, the equation~\eqref{le-pb} is equivalent to 
\begin{equation*} 
\left\{
\begin{aligned}
&i\partial_t v= \Pi\big( W_0(t)  v\big) , \quad   (t,z)\in \R\times \C,\\
&v(0,\cdot)=  u_0 \in \widetilde{\HH}^{s}={\HH}^{\rho s},
\end{aligned}
\right.
\end{equation*}
with the change of unknown $v=e^{it (H+1)^{\rho}}u$. As a consequence we can directly apply the results of Section~\ref{Sect2} (case $\delta=0$) to this model.\medskip

$(i)$ The fact that   $\| \mathscr{L}(t)  \|_{ \mathcal{L}(\widetilde{\HH}^{s})} \leq C_s \eps$ has already been shown in the previous paragraph.\medskip

$(ii)$ For all $s\geq 0$, the problem~\eqref{le-pb} is globally well-posed, in $\widetilde{\HH}^{s}$ by Theorem~\ref{thm2.1}. The group property of $\mathcal{U}$ is a consequence of uniqueness, and its unitarity follows from the conservation of the $L^2$ norm.\medskip

$(iii)$ The upper bound is given by Theorem~\ref{thm2.2}, namely, for all $t\in \R$
 \begin{equation*} 
\| u(t)\|_{\widetilde{\HH}^{s}}   \leq C \|\<z\>^{\rho s} u(t)\|_{L^{2}(\C)} \leq C\|\<z\>^{\rho s} u_0\|_{L^{2}(\C)}  \< \eps t\>^{\rho s}  \leq C\| u_0\|_{\widetilde{\HH}^{s}}    \< \eps t\>^{\rho s},
\end{equation*}
where $\rho= \frac{1}{2(1-\tau)}$. \medskip

$(iv)$  Consider the function $\dis u_0 \in \cap_{k\geq 1} L^{2,k}_{\E}= \cap_{k\geq 1}  \widetilde{\HH}^{k}  $ given by Theorem~\ref{thm2.3}, (see paragraph~\ref{para3.3}) then $\| u_0\|_{\widetilde{\HH}^{s}}  \leq C \sqrt{\eps}$ and 
$$\| u(t)\|_{\widetilde{\HH}^{s}}   \geq c  \| \<z\>^{\rho s} u(t)\|_{L^2(\C)}   \geq c \sqrt{\eps}\<\eps t\>^{\rho s} \geq c\| u_0\|_{\widetilde{\HH}^{s}}  \<\eps t\>^{\rho s} ,$$
hence the result. \medskip

Notice that the items $(ii)$ and $(iii)$ also directly follow from the general result~\cite[Theorem~1.5]{MasRo}.  
 \appendix
 
   \section{A non-perturbative \& time-independent example}\label{appendixA}
   
In this section, we give another example of linear Schr\"odinger operator which yields unbounded dynamics, and which meets  the assumptions (H0)-(H3) of \cite{MasRo}. This example  differs from the one exhibited in Theorem~\ref{thm-abs} in two main aspects :

\begin{itemize}
\item[$\bullet$] it is a non-perturbative example : it is not a lower order perturbation of a time-independent  elliptic differential operator ; 
\item[$\bullet$] it is time-independent.
\end{itemize}

However, with a change of unknown, we can obtain a time-dependent perturbation of a constant coefficient self-adjoint elliptic operator, see Remark~\ref{remchgt} below. \medskip

This example is very simple, and that is why we decided to develop it here. Actually the mechanism involved in the norm inflation is the same as in Theorem~\ref{thm-abs}: it is a traveling wave measured in a weighted $L^2$ space. Actually, our example is  close to the one developed in~\cite[Appendix A]{BGMR2}, after change of variables.  \medskip

On $L^2(\R)$ we define the  usual harmonic oscillator $H=-\partial^2_x +x^2$.  For  $0\leq \tau <1$, we  set $\rho(\tau)=\frac1{2(1-\tau)} \in [1/2 ,\infty)$. We define the operator $\wt{H}=(H+1)^{\rho(\tau)}$ and the  scale  of Hilbert spaces $\big(\widetilde{\HH}^{s}\big)_{s\geq 0}$ by 
\begin{equation*} 
\widetilde{\HH}^{s} = \big\{ u\in L^2(\C),\; \wt{H}^{s/2}u\in L^2(\R)\big\}, \qquad  \widetilde{\HH}^{0} =L^2(\R),
\end{equation*}
endowed with the natural norm  $\|u\|_{\widetilde{\HH}^{s}}    :=   \|  \wt{H}^{s/2}u \|_{L^2(\R)} $.  By \cite[Lemma 2.4]{YajimaZhang2}, we have the following equivalence of norms
\begin{equation} \label{equiv}
\|u\|_{\widetilde{\HH}^{s}}      \equiv \|\<x\>^{\rho s}u\|_{L^2(\R)}+\|(-\partial^2_x)^{\rho s/2}u\|_{L^2(\R)}.
\end{equation}
Now, for $\eps>0$, we consider the problem, 
   \begin{equation} \label{partial}
\left\{
\begin{aligned}
&i\partial_t u=-i\eps \partial_x u , \quad   (t,x)\in \R\times \R,\\
&u(0, \cdot)=  u_0 \in \widetilde{\HH}^{s}.
\end{aligned}
\right.
\end{equation}

In this framework, we are able to prove the following result for the operator $i \eps \partial_x$ in the spaces~$\widetilde{\HH}^{s} $:

  \begin{lem}\label{lem-partial}  
\begin{enumerate}[$(i)$]
 \item  One has $i \eps \partial_x \in  \mathcal{L}(\widetilde{\HH}^{s+1/\rho}, \widetilde{\HH}^{s})\big)$ for all $s \geq 0$, and 
 $$\| i \eps\partial_x \|_{ \mathcal{L}(\widetilde{\HH}^{s+1/\rho}, \widetilde{\HH}^{s})\big)} \leq C_s \eps.$$
  \item  The operator $i \eps \partial_x$ is symmetric  on $ \widetilde{\HH}^{1/\rho}$ w.r.t. the scalar product of $\widetilde{\HH}^{0}$, 
  $$\int_{\R} \ov{v} (i\eps  \partial_x u)  \,dx=  \int_{\R}   u\ov{(i\eps \partial_x v) }\,dx, \quad \forall{u,v \in \widetilde{\HH}^{1/\rho}}.$$
   \item   The   operator $i \eps \partial_x$  is $\wt{H}^{\tau}$-bounded in the sense that 
 $  [i \eps \partial_x, \wt{H}] \wt{H}^{-\tau} \in  \mathcal{L}(\widetilde{\HH}^{s})$ for all $s \geq 0$.
 \end{enumerate}
\end{lem}

Therefore the operator $i \partial_x$ satisfies  the assumptions (H0)-(H3) of \cite{MasRo}. \medskip

On the other hand, we have the following elementary result:

 \begin{prop} \label{thm-abs2} Let $s\geq 0$, then 
 \begin{enumerate}[$(i)$]
\item The problem \eqref{partial} is globally well-posed in $\widetilde{\HH}^{s}$, and the solution is explicitly given by 
$$u(t,x)=u_0(x-\eps t).$$
\item The following bounds hold true : for all $u_0\in \widetilde{\HH}^{s}$ and for all $t\in \R$
 \begin{equation} \label{bound.prop}
c \<\eps t\>^{\frac s{2(1-\tau)}}   \| u_0\|_{\widetilde{\HH}^{s}}  \leq       \| u(t)\|_{\widetilde{\HH}^{s}}\leq C  \<\eps t\>^{\frac s{2(1-\tau)}}\| u_0\|_{\widetilde{\HH}^{s}} .
\end{equation}
 \end{enumerate}
\end{prop}

This result is directly obtained using \eqref{equiv} and the expression $\rho(\tau)=\frac1{2(1-\tau)}$.

\begin{rem}\label{remchgt}
Notice that the function $v(t)=e^{-it \wt{H}}u(t)$ is solution to
   \begin{equation*} 
\left\{
\begin{aligned}
&i\partial_t v-\wt{H}v=-i \eps \big(e^{-it \wt{H}}  \partial_x e^{it \wt{H}}\big)v , \quad   (t,x)\in \R\times \R,\\
&v(0, \cdot)=  v_0=u_0 \in \widetilde{\HH}^{s},
\end{aligned}
\right.
\end{equation*} 
and satisfies the conclusions of Lemma \ref{lem-partial} and the bounds~\eqref{bound.prop}. This yields an example in the spirit of the one exhibited in Theorem~\ref{thm-abs}, but in the present case, the perturbation is of order~$1/\rho \in (0,2]$ instead of being of order 0. 
\end{rem}

\begin{proof}[Proof of Lemma~\ref{lem-partial}] Item $(i)$ is a direct consequence of~\eqref{equiv} and  $(ii)$  is elementary. 

$(iii)$ It is convenient to introduce the Sobolev space based on the harmonic oscillator ($s\geq 0$)
\begin{equation*} 
{\HH}^{s} = \big\{ u\in L^2(\R),\; H^{s/2}u\in L^2(\R)\big\}, \qquad  {\HH}^{0} =L^2(\R).
\end{equation*}
Thanks to the pseudo-differential calculus associated to $H$ (see also paragraph~\ref{para4.3}), we first prove that    for all $r>0$,
\begin{equation}\label{com-r} 
    [i \partial_x, (H+1)^r  ]  (H+1)^{-r+\frac12}   \in  \mathcal{L}\big(L^2(\R)\big).
    \end{equation}
Here the symbol class $S^m$ reads
\begin{equation*} 
S^m =\Big\{ a \in \mathcal{C}^\infty(\R^2; \C)  :  \; \big|\partial_{x}^{\alpha} \partial_{\xi}^{\beta}  a(x, \xi)\big|
\leq C_{\alpha,\beta}\langle |x|+|\xi|\rangle^{m-\beta}, \;\;  \forall \, \alpha, \beta \in \N\Big\}.
\end{equation*}
The  symbol of $    [i \partial_x, (H+1)^r  ]  $, modulo terms in $S^{2r-1}$, is given by the formula
\begin{eqnarray*} 
-i \big\{ \xi , (x^2 +\xi^2+1)^{r} \big\} &=&-i\partial_{\xi}  (\xi) \partial_{x} \big(  (x^2 +\xi^2+1)^{r} \big)+i\partial_{x}  (\xi) \partial_{\xi} \big(  (x^2 +\xi^2+1)^{r} \big) \\
&=&-i  2r x (x^2 +\xi^2+1)^{r-1}  \in S^{2r-1}.
\end{eqnarray*} 
Since the symbol of $ (H+1)^{-r+\frac12} $ belongs to $S^{-2r+1}$, we deduce \eqref{com-r}. \medskip

 Next, for $\rho>0$ and $s \geq 0$
\begin{multline*}
 (H+1)^{\frac{s}2}   [i \partial_x, (H+1)^{\rho}  ]  (H+1)^{-\rho+\frac12} (H+1)^{-\frac{s}2}=\\
 =- [i \partial_x, (H+1)^\frac{s}2  ] (H+1)^{-\frac{s}2  +\frac12} + [i \partial_x, (H+1)^{\frac{s}2+\rho}  ] (H+1)^{-(\frac{s}2  +\rho)+\frac12}, 
 \end{multline*}
and by applying \eqref{com-r} twice, we deduce that for all $s\geq 0$
\begin{equation}\label{com-q} 
    [i \partial_x, (H+1)^{\rho}  ]  (H+1)^{-\rho+\frac12}   \in  \mathcal{L}\big({\HH}^{s}\big).
    \end{equation}
    Finally recall that $\wt{H}=(H+1)^{\rho}$, thus \eqref{com-q} is equivalent to 
    \begin{equation*} 
    [i \partial_x,  \wt{H}  ]  \wt{H}^{- \tau }   \in  \mathcal{L}\big(\wt{\HH}^{s}\big),
    \end{equation*}
since  $\tau = 1-1/(2 \rho)$, which was the claim.
\end{proof}


  \section{On the reducibility of the linear LLL equation}\label{appendixB}
 
 We state here a reducibility result for the linear LLL equation. It turns out that the abstract reducibility result obtained in~\cite{Grebert-Thomann} can be applied to this model, which is close in many aspects to the  usual 1D cubic quantum harmonic oscillator with time-dependent potential. We consider the linear equation
  \begin{equation}\label{linH-KAM} 
\left\{
\begin{aligned}
&i\partial_t u-\delta Hu=\eps \Pi(W(t \om,z)u), \quad   (t,z)\in \R\times \C,\\
&u(0,z)=  u_0(z),
\end{aligned}
\right.
\end{equation}
where $\delta \neq 0$, where $\epsilon >0$ is small and where  the parameter    $\omega \in [0,2\pi)^n$ is the frequency vector, for some given $n \geq 1$. Up to a rescaling, we can assume that $\delta =1$. We   assume in the sequel that the potential
 \begin{equation*}
 \begin{array}{rccl}
W :  &\T^n\times \C&\longrightarrow&\R \qquad \qquad\qquad \T^n: =(\R/2\pi \Z)^n\\[3pt]
\dis & (\theta, z) &\longmapsto &  W(\theta,z),
 \end{array}
 \end{equation*}
is analytic in $\theta$ on $|\text{Im}\,\theta|<\tau$ for some $\tau>0$, and $\mathcal{C}^{2}$ in $x,y$ (where $z=x+iy$), and we suppose moreover that there exists $\gamma>0$ and $C>0$ so that for all $\theta \in \T^{n}$ and $z\in \C$ 
\begin{equation}\label{cond-W}
 |W(\theta,z)|\leq C\<z\>^{-\gamma},\qquad\; |\partial^j_{z}\partial^{\ell}_{\ov z}W(\theta,z)|\leq C,
 \end{equation}
 for any $0 \leq j, \ell \leq 1$. \medskip
 
 When $\om =0$, all the solutions to \eqref{linH-KAM} are almost periodic in time. This can be proved by constructing a Hilbertian basis\footnote{Such a Hilbertian basis exists, since $u \mapsto Hu+\eps \Pi(W(0,z)u)$ is a self-adjoint operator with compact resolvent in~$\E$.} $(\psi_k)_{k \geq 0}$ of $\E$ composed   of eigenfunctions of the   operator $u \mapsto Hu+\eps \Pi(W(0,z)u)$, such that 
 $$H\psi_k+\eps \Pi\big(W(0,z)\psi_k\big)=\lambda_k \psi_k, \quad k \geq 0.$$
 Then \eqref{linH-KAM} can be solved by
 $$   u(t,z)= \sum_{k=0}^{+\infty}c_k e^{-it \lambda_k } \psi_k(z), \qquad u_0(z)= \sum_{k=0}^{+\infty}c_k   \psi_k(z),$$
which shows that any solution to \eqref{linH-KAM} is an infinite superposition of periodic functions, hence it is an almost-periodic function in time. \medskip

For $\om \neq 0$, the reducibility theory adresses the question if, by the means of a  time quasi-periodic transformation, one can reduce to the previous case. It turns out that for \eqref{linH-KAM}, it is the case for a large set of values $\om \in \Lambda_{\eps} $ :

 \begin{thm}\label{theo-lin-LLL} Assume that $W$ satisfies \eqref{cond-W}. Then there exists $\epsilon_0$ such that for all $0\leq\epsilon<\epsilon_0$ there exists a set  $\Lambda_{\eps} \subset [0,2\pi)^n$ of positive measure and asymptotically full measure: $\mbox{Meas}(\Lambda_{\eps} ) \to (2\pi)^n$ as $\epsilon \to 0$, such that for all 
$\omega\in \Lambda_{\eps} $, the linear   equation \eqref{linH-KAM} reduces, in $\E$, to a  linear equation with constant coefficients.
  \end{thm}
  
  We refer to \cite[Theorem 7.1]{Grebert-Thomann} for a more precise statement, giving in particular more information on the transformation. \medskip

Assume that $(\theta, z) \mapsto V(\theta,z)$ is analytic in $\theta$ on $|\text{Im}\,\theta|<\tau$,  that $V(\theta, \cdot) \in \E$ for all $\theta \in \T^n$, and satisfies, for some $\gamma>0$, the bound  $|V(\theta,z)|\leq C\<z\>^{-\gamma}$ uniformly in $\theta \in \T^n$. Then, by Lemma~\ref{lem.deri},  $W=|V|^2$ satisfies \eqref{cond-W}. Such a potential even satisfies $|\partial^j_{z}\partial^{\ell}_{\ov z}W(\theta,z)|\leq C$ for all $k, \ell \in \N$ (without additional assumptions on $V \in \E$). \medskip

We also have the   following result on  the dynamics of the solutions of \eqref{linH-KAM} :
  \begin{cor} 
  Assume that $W$ is  $\mathcal{C}^{\infty}$ in $x,y$ with all its derivatives bounded  and satisfying  \eqref{cond-W}. Let $s\geq 0$ and $u_{0}\in \HH^{s}$. Then there exists $\eps_{0}>0$ so that for all $0<\eps<\eps_{0}$ and $\om \in \Lambda_{\eps}$,    there exists a unique solution $u \in \mathcal{C}\big(\R\,;\,\HH^{s}\big)$ of \eqref{linH-KAM} so that $u(0)=u_{0}$. Moreover, $u$ is almost-periodic in time and we have the bounds 
  \begin{equation*}
  (1-\eps C)\|u_{0}\|_{\HH^{s}}\leq \|u(t)\|_{\HH^{s}}\leq  (1+\eps C)\|u_{0}\|_{\HH^{s}}, \qquad \forall \,t\in \R,
  \end{equation*}
  for some $C=C(s,\om)$.
  \end{cor}
The result of Theorem \ref{theo-lin-LLL} can also be formulated in term of  the Floquet operator.  Consider  the Floquet Hamiltonian operator, defined on $\E\otimes L^2(\T^n)$ by 
 \begin{equation*} 
 K:=i\sum_{k=1}^n\omega_k  \partial_{ \theta_k} +H +\epsilon \Pi \big(W(\theta,z) \cdot  \big),
\end{equation*}
then we can state  
 \begin{cor}
 Assume that $W$ satisfies  \eqref{cond-W}. There exists $\eps_{0}>0$ so that for all $0<\eps<\eps_{0}$ and $\om \in \Lambda_{\eps}$, the spectrum of the Floquet operator $K$ is pure point.
 \end{cor}

We refer to \cite[Section~7]{Grebert-Thomann}, where similar results are proven for the 1D quantum harmonic oscillator. \medskip

For the reducibility of the periodic Schr\"odinger equation, we refer to \cite{ElKuk2} and for the reducibility of the quantum harmonic oscillator in any dimension to~\cite{GP,LW,BGMR2} and we refer to \cite{Bambusi1, Bambusi2, Bambusi3} for the reducibility for 1-$d$ operators with unbounded perturbations. For references on the theory of Floquet operators, see  \cite{Eli, Wang}.\medskip

Finally, let us mention that, concerning the nonlinear  cubic LLL equation,  the abstract KAM result of \cite{Grebert-Thomann} was applied in \cite[Theorem 4.3]{GGT} in order to show the existence of invariant torii, and that the result of \cite{GIP} was applied to show an almost global existence result for the cubic LLL equation. We refer to \cite[Section~4.2]{GGT} for more details. \medskip

 The arguments of~\cite[Section~7]{Grebert-Thomann} can be directly  applied to the equation \eqref{linH-KAM}, and we address the reader to this latter paper for the proofs of the previous results.  Let us just sketch the idea : we expand~$u$ and~$\bar u$ on the basis given by the special Hermite functions
 $$\dis u=\sum_{j\geq 0} c_j {\phi_j}, \quad  \dis \ov{u}=\sum_{j\geq 0} \ov{c}_j \ov{\phi}_j.$$
 Then equation  \eqref{linH-KAM} reads as  an autonomous Hamiltonian system in an extended phase space
\begin{equation}\label{liH}
\left\{ \begin{array}{ll}     
\dot c_j=-2i(j+1) c_j- i\eps \partial_{\bar c_j}Q(\theta,c,\bar c)\quad & j\geq 0 \\[5pt]
\dot{\bar c}_j= 2i(j+1)\bar c_j+ i\eps \partial_{  c_j}Q(\theta,c,\bar c) & j\geq 0\\[5pt]
\dot \theta_j= \omega_j & j=0,\dots, n\\[5pt]
\dot Y_j= -\eps \partial_{ \theta_j}Q(\theta,z,\bar z) & j=0,\dots, n
\end{array}\right.
\end{equation} 
where $Q$  is a quadratic functional in $(c,\bar c)$ given by
\begin{equation}\label{def-Q}
Q(\theta,c,\bar c)=\int_\C W(\theta,z)\Big(\sum_{j\geq 0} c_k \phi_k(z)\Big) \ov{\Big(\sum_{j\geq 0}  c_j {\phi_j(z)}\Big)}dL(z),
\end{equation}
and the Hamiltonian of the system~\eqref{liH} is 
$$
\sum_{j=1}^n \omega_j Y_j +2\sum_{j\geq 0} (j+1) c_j\ov{c}_j+Q(\theta,c,\bar c).
$$
Then, we can check that \eqref{def-Q} satisfies the assumptions of \cite[Theorem 7.1]{Grebert-Thomann}. The dispersive estimate $\|\phi_n \|_{L^\infty(\C)} \leq C n^{-1/4}$  satisfied by the   $(\phi_n)_{n\geq 0}$, is the key ingredient which    allows  to follow the lines of~\cite[Section~7]{Grebert-Thomann}.

  \section{Some technical results}

 \begin{lem}\label{lem-comm}
 The operators $H$ and $\Pi$ commute.
\end{lem}

\begin{proof}
Recall that 
$$
[\Pi u](z) = \frac{1}{\pi} e^{-\frac{|z|^2}{2}} \int_\mathbb{C} e^{\ov  w z - \frac{|w|^2}{2}} u(w) \,dL(w),
$$
and that 
$$H=-4\partial_z \partial_{\ov z}+|z|^2.$$
On the one hand, by integration by parts
\begin{eqnarray*}
\big(\Pi \partial_z\partial_{\ov z} u\big)(z)&=& \frac{1}{\pi} e^{-\frac{|z|^2}{2}} \int_\mathbb{C} e^{\ov  w z - \frac{|w|^2}{2}} \partial_w\partial_{\ov w} u(w) \,dL(w)\\
&=& \frac{1}{\pi} e^{-\frac{|z|^2}{2}} \int_\mathbb{C}  \partial_w\partial_{\ov w} \big( e^{\ov  w z - \frac{|w|^2}{2}}\big)  u(w) \,dL(w)\\
&=& -\frac12 \Pi u(z)-\frac12 z \Pi(\ov{w}u)(z)+\frac14 \Pi(|w|^2u)(z),
\end{eqnarray*}
thus
$$ \Pi Hu(z)=2\Pi u(z)+2z  \Pi(\ov{w}u)(z).$$
On the other hand 
\begin{eqnarray*}
\partial_z  \Pi u(z) &=&-  \frac{1}{\pi}\frac{\ov{z}}2 e^{-\frac{|z|^2}{2}} \int_\mathbb{C} e^{\ov  w z - \frac{|w|^2}{2}} u(w) \,dL(w)+ \frac{1}{\pi}  e^{-\frac{|z|^2}{2}} \int_\mathbb{C} e^{\ov  w z - \frac{|w|^2}{2}} \ov{w}u(w) \,dL(w)\\
&=&-\frac{\ov{z}}{2}\Pi u(z)+\Pi\big(\ov{z}u\big)(z),
\end{eqnarray*}
then 
$$  \partial_{\ov z}\partial_z \Pi u(z)=-\frac{1}2 \Pi u(z)+ \frac14 |z|^2 \Pi u(z)-\frac{z}2\Pi(\ov{w}u)(z) ,$$
and we get that 
$$H \Pi u(z)=2\Pi u(z)+2z  \Pi(\ov{w}u)(z) =\Pi Hu(z),$$
which was the claim.
\end{proof}

We recall a short version of \cite[Lemma~A.2]{Schw-Tho} :
\begin{lem}\label{lem.deri}
For all $ j, k \geq 0$ there exists $C>0$ such that for all  $1\leq p \leq \infty$ and $v\in \E$, 
$$ \big\| \partial^j_{\ov z}\partial^k_z\big(|v|^2\big)\big\|_{L^p(\C)} \leq C   \| v\|^2_{L^{2p}(\C)}.$$
\end{lem}

\end{document}